\documentclass[dvipdfmx]{amsart}
\usepackage{amssymb, amsmath, amscd,latexsym, amsthm, amscd}
\usepackage{graphicx}
\usepackage[abs]{overpic} %
\usepackage{color}
\usepackage{here}
\usepackage[all]{xy}

\newtheorem{thm}{Theorem}[section]

{\theoremstyle{definition} 
\newtheorem{rem}[thm]{Remark}
}
\newtheorem{question}[thm]{Question}
\newtheorem{prop}[thm]{Proposition}

\begin{document}
%
%
%
%
%
%
%
\title
{Notes on constructions of knots with the same trace}
\author{Keiji Tagami}
\address{Department of Fisheries Distribution and Management, 
National Fisheries University, 
Shimonoseki, Yamaguchi 759-6595 
JAPAN
}
\email{tagami@fish-u.ac.jp}
\date{\today}
\maketitle
%
\begin{abstract}
The $m$-trace of a knot is the $4$-manifold obtained from $\mathbf{B}^4$ by attaching a $2$-handle along the knot with $m$-framing. 
In 2015, Abe, Jong, Luecke and Osoinach introduced a technique to construct infinitely many knots with the same $m$-trace, which is called the operation $(\ast m)$. 
In this paper, we prove that their technique can be explained in terms of Gompf and Miyazaki's dualizable pattern. 
In addition, we show that the family of knots admitting the same $4$-surgery given by Teragaito can be explained by the operation $(\ast m)$.  
\end{abstract}
\section{Introduction}
For an integer $m$, the {\it $m$-trace} $X_{K}(m)$ of a knot $K$ is the $4$-manifold obtained from $\mathbf{B}^4$ by attaching a $2$-handle along the knot with $m$-framing. 
On techniques to construct knots with the same trace, the following are known. 
\begin{itemize}
\item Abe, Jong, Luecke and Osoinach \cite{AJLO} introduced a technique to construct infinitely many knots with the same $m$-trace, which is based on ``annulus presentation" and called {\it the operation $(\ast m)$} 
\footnote{
In \cite[Section~3.1.2]{AJLO}, we use $(\ast n)$ instead of $(\ast m)$.
}
(for annulus presentations and the operation $(\ast m)$, see Section~\ref{sec:annulus}). 
\item Miller and Piccirillo \cite{Miller-Piccirillo} constructed a pair of knots with the same $m$-trace by utilizing dualizable patterns (for dualizable patterns, see Section~\ref{sec:dualizable1}). 
\end{itemize}
Miller and Piccirillo \cite{Miller-Piccirillo} explained the operation $(\ast m)$ in terms of dualizable patterns in the case $m=0$ by constructing a dualizable pattern from an annulus presentation with some condition. 
\par 
In this paper, we proved that, for any $m\in \mathbf{Z}$, the operation $(\ast m)$ can be also explained in terms of dualizable patterns (Theorem~\ref{thm:star-dualizable}). 
As an application, we directly draw the duals to Miller and Piccirillo's dualizable patterns 
(Theorem~\ref{thm:main2} and Figure~\ref{figure:dual}). 
In addition, we explain the family of knots admitting the same $4$-surgery given by Teragaito \cite{Teragaito} in terms of the operation $(\ast m)$ (Section~\ref{sec:Teragaito}). 
We also some observations in the final section. 
Throughout this paper, 
\begin{itemize}
\item unless specifically mentioned, all knots and links are smooth and unoriented, and all other manifolds are smooth and oriented, 
\item for an $n$-component link $L_{1}\cup\dots \cup L_{n}$, we denote the $3$-manifold obtained from $\mathbf{S}^3$ by $m_i$-surgery on a knot $L_i$ for any $i$ by $M_{L_{1}\cup\dots \cup L_{n}}(m_1,\dots, m_n)$, 
\item we denote a tubular neighborhood of a knot $K$ in a $3$-manifold by $\nu(K)$, 
\item we denote the unknot in $\mathbf{S}^3$ by $U$. 
\end{itemize}
\section{Annulus twist, annlus presentation and the operation $(\ast m)$}\label{sec:annulus}
\subsection{Annulus twist and annulus presentation}\label{sec:annulus-twist-pre}
Let $A\subset \mathbf{S}^{3}$ be an embedded annulus with ordered boundaries $\partial A=c_{1}\cup c_{2}$. 
An {\it $n$-fold annulus twist along $A$} is to apply $(lk(c_1,c_2)+1/n)$-surgery along $c_{1}$ and $(lk(c_1,c_2)-1/n)$-surgery along $c_{2}$,  where we give $c_1$ and $c_2$ parallel orientations. 
We see that the resulting manifold obtained by an annulus twist is also $\mathbf{S}^3$. 
\par
Let $A\subset \mathbf{S}^{3}$ be an embedded annulus with $\partial A=c_{1}\cup c_{2}$. 
Take an embedding of a band $b\colon I\times I\rightarrow \mathbf{S}^{3}$ such that 
\begin{itemize}
\item $b(I\times I)\cap \partial A=b(\partial I\times I)$, 
\item $b(I\times I)\cap \operatorname{Int} A$ consists of ribbon singularities, and 
\item $A\cup b(I\times I)$ is an immersion of an orientable surface, 
\end{itemize}
where $I=[0, 1]$. 
If a knot $K\subset \mathbf{S}^3$ is isotopic to the knot $(\partial A\setminus b(\partial I\times I))\cup b(I\times \partial I)$, 
then we call $(A, b)$ an {\it annulus presentation} of $K$. 
An annulus presentation $(A,b)$ is {\it special} if $A$ is unknotted and $\operatorname{lk}(c_1, c_2)=\pm 1$ (that is, $A$ is $\pm 1$-full twisted). 
Let $K$ be a knot with an annulus presentation $(A, b)$. 
Let $A'\subset A$ be a shrunken annulus with $\partial A'=c'_{1}\cup c'_{2}$ which satisfies the following: 
\begin{itemize}
\item $\overline{A\setminus A'}$ is a disjoint union of two annuli, 
\item each $c'_{i}$ is isotopic to $c_{i}$ in $\overline{A\setminus A'}$ for $i=1,2$ and 
\item $A\setminus (\partial A\cup A')$ does not intersect $b(I\times I)$. 
\end{itemize}
Then, by $A^{n}(K)$, 
we denote the knot obtained from $K$ by the $n$-fold annulus twist along $A'$. 
For simplicity, we denote $A^{1}(K)$ by $A(K)$ and $A^{0}(K)$ by $K$. 
\begin{rem}
We find many examples of annulus presentations in \cite{AJOT, AJLO, abe-tagami2, tagami6}. 
In this paper, for an annulus presentation $(A,b)$, we often draw the attaching regions $A\cap b$ by bold arcs and we omit the band $b$. 
\end{rem}
\par
By utilizing Osoinach's work \cite[Theorem 2.3]{Osoinach}, for a knot $K$ with an annulus presentation $(A, b)$, we see that 
$M_{K}(0)$ and $M_{A^{n}(K)}(0)$ are orientation-preservingly homeomorphic for any $n\in\mathbf{Z}$. 
In particular, a homeomorphism $\phi_{n}\colon M_{K}(0) \rightarrow M_{A^{n}(K)}(0)$ is given as in Figure~\ref{figure:Osoinach-homeo}, which is explicitly given by Teragaito \cite{Teragaito}. 
We call $\phi_{n}$ the $n$-th Osoinach-Teragaito's homeomorphism. 
%
Moreover, if $(A,b)$ is special, by applying Abe, Jong, Omae and Takeuchi's result \cite[Theorem~2.8]{AJOT} to the knot, we see that 
the homeomorphism $\phi_{n}$ extends to an orientation-preserving diffeomorphism $\Phi_{n}\colon X_{K}(0)\rightarrow X_{A^{n}(K)}(0)$ for any $n\in\mathbf{Z}$. 
\par
As a consequence, we obtain the following. 
\begin{thm}\label{thm:Osoinach} 
Let $K\subset \mathbf{S}^3$ be a knot with an annulus presentation $(A, b)$. 
Then, there is an orientation-preservingly homeomorphism $\phi_{n}\colon M_{K}(0) \rightarrow M_{A^{n}(K)}(0)$ for any $n\in\mathbf{Z}$. 
In particular, $\phi_{n}$ is given as in Figure~\ref{figure:Osoinach-homeo}. 
Moreover, if $(A,b)$ is special, $\phi_{n}$ extends to an orientation-preserving diffeomorphism $\Phi_{n}\colon X_{K}(0)\rightarrow X_{A^{n}(K)}(0)$. 
\end{thm}
\begin{figure}[h]
\centering
\includegraphics[scale=0.7]{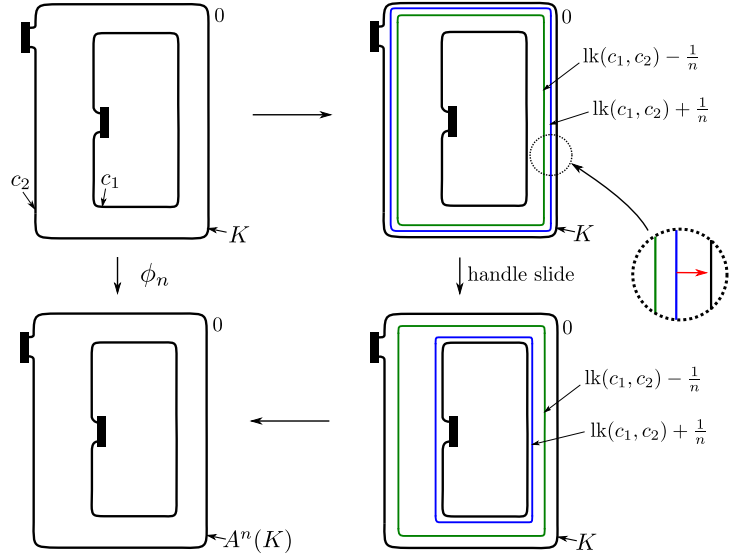}
\caption{
(color online) Osoinach-Teragaito's homeomorphism $\phi_{n}$. For simplicity we draw $A$ as a flat annulus although $A$ may be knotted and twisted. }
\label{figure:Osoinach-homeo}
\end{figure}
\subsection{Operation $(\ast m)$}
%
Let $K$ be a knot with a special annulus presentation $(A,b)$. 
Let $\gamma_{A(K)}\subset \mathbf{S}^3\setminus \nu(A(K))$ be a curve depicted in Figure~\ref{figure:gamma}. 
Remark that the definition of $\gamma_{A(K)}$ depends on the twist of $A$. 
\begin{figure}[h]
\centering
\includegraphics[scale=0.7]{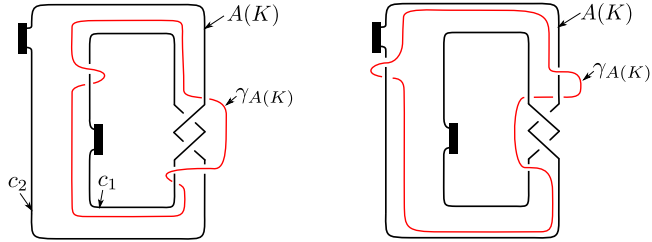}
\caption{(color online) The curve $\gamma_{A(K)}$ 
}\label{figure:gamma}
\end{figure}
Denote the knot obtained from $A(K)$ by twisting $m$ times along $\gamma_{A(K)}$ by $T_{m}(A(K))$. 
In \cite[Section~3.1.2]{AJLO}, the operation $K\mapsto T_{m}(A(K))$ is called {\it the operation $(\ast m)$}. 
Then, Abe, Jong, Luecke and Osoinach \cite{AJLO} proved the following theorem. 
\begin{thm}[{\cite[Theorem~3.7 and Theorem~3.10]{AJLO}}]\label{thm:AJLO1}
Let $K$ be a knot with a special annulus presentation $(A,b)$. 
Then, there is an orientation-preservingly homeomorphism $\psi_{m}\colon M_{K}(m)\rightarrow M_{T_{m}(A(K))}(m)$ which extends to a diffeomorphism $\Psi_{m}\colon X_{K}(m)\rightarrow X_{T_{m}(A(K))}(m)$ for any $m\in\mathbf{Z}$. 
\end{thm}
Concretely, $\psi_{m}$ is given as in Figure~\ref{figure:AJLO-homeo} for the case $A$ is $+1$ twisted. 
Similarly, we also obtain $\psi_{m}$ for the case $A$ is $-1$ twisted. 
%
%
\begin{rem}
Note that Osoinach-Teragaito's homeomorphism induces a homeomorphism $\phi_{+1}\colon (M_{K}(0), \alpha_{K})\rightarrow (M_{A(K)}(0), \gamma_{A(K)})$, where $\alpha_{K} \subset \mathbf{S}^{3}\setminus \nu(K)$ is a meridian of $K$ and we regard $\alpha_{K}$ and $\gamma_{A(K)}$ as curves in $M_{K}(0)$ and $M_{A(K)}(0)$ 
(see also the bottom arrow in Figure~\ref{figure:AJLO-homeo}). 
\end{rem}
\begin{figure}[h]
\centering
\includegraphics[scale=0.72]{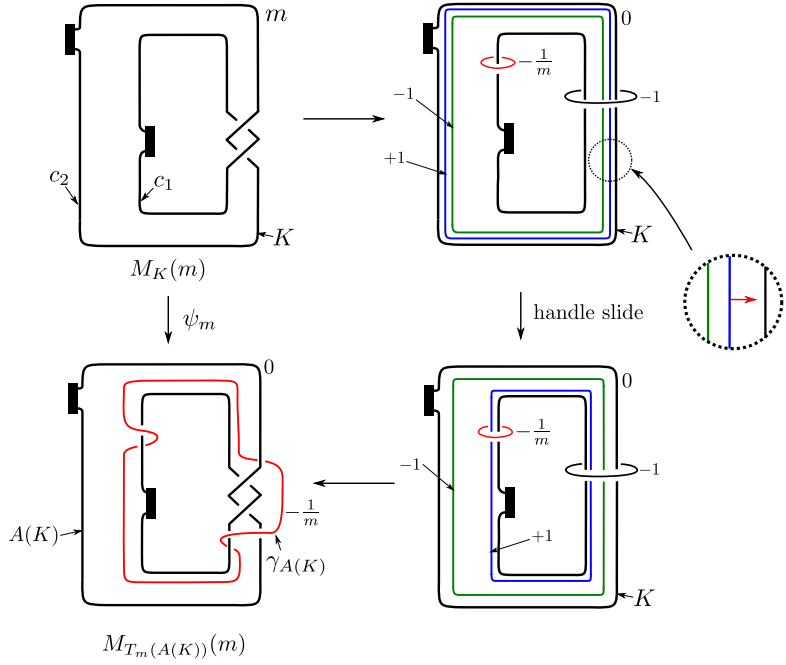}
\caption{
(color online) The homeomorphism $\psi_{m}\colon M_{K}(m)\rightarrow M_{T_{m}(A(K))}(m)$ for the case $A$ is $+1$ twisted. 
For the case $A$ is $-1$ twisted, we define $\psi_{m}$ by the same way. 
}
\label{figure:AJLO-homeo}
\end{figure}
\section{Relation between annulus presentation and dualizable pattern}\label{sec:dualizable1}
\subsection{Dualizable pattern}
Here, we recall the definition of dualizable patterns, which is firstly given by Gompf and Miyazaki \cite{Gompf-Miyazaki} and developed by Miller and Piccirillo \cite{Miller-Piccirillo} (see also \cite{tagami6}). 
\par 
Let $P\colon \mathbf{S}^1\rightarrow V$ be an oriented knot in a solid torus $V=\mathbf{S}^1\times D^2$. 
Suppose that the image $P(\mathbf{S}^1)$ is not null-homologous in $V$. 
Such a $P$ is called a {\it pattern}. 
By an abuse of notation, we use the notation $P$ for both a map and its image. 
Define $\lambda_{V}$, $\mu_{P}$, $\mu_{V}$ and $\lambda_{P}$ as follows: 
\begin{itemize}
\item put $\lambda_{V}=\mathbf{S}^1\times \{x_0\}\subset \partial V \subset V$ for some $x_0\in \partial D^{2}$ and orient $\lambda_{V}$ so that $P$ is homologous to $r\lambda_{V}$ in $V$ for some positive $r\in\mathbf{Z}_{>0}$, 
\item define $\mu_{P}\subset V$ by a meridian of $P$ and orient $\mu_{P}$ so that the linking number of $P$ and $\mu_{P}$ is $1$, 
\item put $\mu_{V}=\{x_1\}\times \partial D^2\subset \partial V \subset V$ for some $x_{1}\in \mathbf{S}^1$ and orient $\mu_{V}$ so that $\mu_{V}$ is homologous to $s\mu_{P}$ in $V\setminus \nu(P)$ for some positive $s\in\mathbf{Z}_{>0}$, 
\item define $\lambda_{P}$ by a longitude of $P$ which is homologous to $t\lambda_{V}$ in $V\setminus \nu(P)$  for some positive $t\in\mathbf{Z}_{>0}$. 
\end{itemize}
\par
For an oriented knot $K\subset \mathbf{S}^3$, let $\iota_{K}\colon V\rightarrow \mathbf{S}^3$ be an embedding which identifies $V$ with $\overline{\nu(K)}$ and sends $\lambda_{V}$ to an oriented curve on $\partial \overline{\nu(K)}$ which is null-homologous in $\mathbf{S}^3\setminus \nu(K)$ and isotopic to $K$ in $\mathbf{S}^3$. 
Then $\iota_{K}\circ P\colon \mathbf{S}^1\rightarrow \mathbf{S}^3$ represents an oriented knot. 
The knot is called the {\it satellite} of $K$ with pattern $P$ and denoted by $P(K)$. 
\par
A pattern $P\colon \mathbf{S}^{1}\rightarrow V$ is {\it dualizable} if there is a pattern $P^{\ast}\colon \mathbf{S}^{1}\rightarrow V^{\ast}$ and an orientation-preserving homeomorphism $f\colon V\setminus \nu(P)\rightarrow V^{\ast}\setminus \nu(P^{\ast})$ such that 
$f(\lambda_{V})=\lambda_{P^{\ast}}$, $f(\lambda_{P})=\lambda_{V^{\ast}}$, $f(\mu_{V})=-\mu_{P^{\ast}}$ and $f(\mu_{P})=-\mu_{V^{\ast}}$. 
\par
Miller and Piccirillo \cite[Proposition~2.5]{Miller-Piccirillo} introduced a convenient technique to determine whether a given pattern is dualizable as follows (see also \cite[Section~2]{Gompf-Miyazaki}). 
Define $\Gamma\colon \mathbf{S}^1\times D^{2} \rightarrow \mathbf{S}^1\times \mathbf{S}^2$ by $\Gamma(t,d)=(t,\gamma(d))$, where $\gamma\colon D^{2}\rightarrow \mathbf{S}^2$ is an arbitrary orientation preserving embedding. 
For any curve $c\colon \mathbf{S}^1\rightarrow \mathbf{S}^1\times D^{2}$, define $\widehat{c}=\Gamma\circ c \colon \mathbf{S}^1\rightarrow \mathbf{S}^1\times \mathbf{S}^2$. 
Then, we obtain the following proposition. 
\begin{prop}[{\cite[Proposition~2.5]{Miller-Piccirillo}}]\label{prop:MP}
A pattern $P$ in a solid torus $V$ is dualizable if and only if $\widehat{P}$ is isotopic to $\widehat{\lambda_{V}}$ in $\mathbf{S}^{1}\times \mathbf{S}^2$. 
\end{prop}
%
%
%
%
%
\par
Related to knot traces, the following are known. 
Let $P\subset V$ be a pattern. 
Let $\tau_{m}\colon V\rightarrow V$ be a homeomorphism given by twisting $m$ times along a meridian of $V$. 
It is known that if $P$ is dualizable then $\tau_m(P)$ is also dualizable and its dual is given by $\tau_{-m}(P^{\ast})$, where $P^{\ast}$ is the dual to $P$ (see \cite[Theorem~3.6]{Miller-Piccirillo} and \cite[Remark~4.6]{tagami6}). 
Moreover, we obtain the following. 
\begin{thm}[{\cite[Theorem~3.6]{Miller-Piccirillo} and \cite[Remark~4.6]{tagami6}}]\label{thm:MP}
Let $P$ be a dualizable pattern and $P^{\ast}$ be its dual. 
Then, we have $X_{P(U)}(m)\cong X_{\tau_{m}(P^{\ast})(U)}(m)$ for any $m\in \mathbf{Z}$. 
\end{thm}
%



%
%
\begin{rem}
For a dualizable pattern $P\subset V$, we see that $M_{P(U)\cup \mu _{V}}(0,0)\cong \mathbf{S}^3$. 
Conversely, for a link $k\cup c$ with $M_{k\cup c}(0,0)\cong \mathbf{S}^3$, we see that $k\subset \mathbf{S}^3\setminus \nu(c)$ is a dualizable pattern after giving some orientation to $k$ (for detail, see \cite{Baker-Motegi} and \cite[Remarks~3.3 and 4.6]{tagami6}). 

\end{rem}
%
%
\subsection{From special annulus presentations to dualizable patterns}\label{sec:annulus-dualizable}
In this section, we recall Miller and Piccirillo's construction (\cite[Section~5]{Miller-Piccirillo}) of dualizable patterns from a special annulus presentation (see also \cite{tagami6}). 
\par 
Let $K\subset \mathbf{S}^3$ be a knot with a special annulus presentation $(A, b)$. 
In Figure~\ref{figure:annulus-dualizable}, the left knots represent $K$, and each right knot represents $A^{\pm1}(K)$ for the corresponding left $K$. 
Then, for each case, take curves $\beta_{K}^{\pm} \subset \mathbf{S}^{3}\setminus \nu(K)$ as in Figure~\ref{figure:annulus-dualizable}. 
\par
Let $P_{+}$ (resp. $P_{-}$) be the pattern given by $K\subset V_{+}=\mathbf{S}^3\setminus \nu(\beta_{K}^{+})$ (resp. $K\subset V_{-}=\mathbf{S}^3\setminus \nu(\beta_{K}^{-})$), where we give a parameter of $V_{\pm}$ so that  $P_{\pm}(U)=K$. 
Moreover, we give an orientation of $P_{\pm}$ arbitrarily. 
Then, we can check that $P_{\pm}$ are dualizable patterns 
(for example, slide $K$ along the $0$-framing of $\beta_{K}^{\pm}$ in $M_{\beta_{K}^{\pm}}(0)\cong \mathbf{S}^1\times \mathbf{S}^2$ and apply Proposition~\ref{prop:MP}). 
These dualizable patterns satisfy the following. 
\begin{prop}[{e.g. \cite[Proposition~5.3]{Miller-Piccirillo} and \cite[Proposition~3.9]{tagami6}}]\label{prop:Miller-Piccirillo}
Let $K$ be a knot with a special annulus presentation $(A,b)$. 
Let $P_{+}$ and $P_{-}$ be the dualizable patterns as above. Then we have $P_{\pm}(U)=K$ and $P_{\pm}^{\ast}(U)=A^{\pm 1}(K)$. 
\end{prop}
\begin{rem}\label{rem:regarding}
The homeomorphisms given in Figure~\ref{figure:Osoinach-homeo} induces homeomorphisms 
\begin{align*}
&\phi_{\pm1}\colon (M_{K}(0), \beta_{K}^{\pm})\rightarrow (M_{A^{\pm1}(K)}(0), \alpha_{A^{\pm1}(K)}), 
\end{align*}
where $\alpha_{A^{\pm1}(K)} \subset \mathbf{S}^{3}\setminus \nu(A^{\pm1}(K))$ is a meridian of $A^{\pm1}(K)$. 
Here we regard $\beta_{K}^{\pm}$ and $\alpha_{A^{\pm1}(K)}$ as curves in $M_{K}(0)$ and $M_{A^{\pm1}(K)}(0)$ under the identifications 
\begin{align*}
\mathbf{S}^{3}\setminus \nu(K)&\cong M_{K}(0)\setminus \nu(L_{K}), \\
\mathbf{S}^{3}\setminus \nu(A^{\pm1}(K))&\cong M_{A^{\pm1}(K)}(0)\setminus \nu(L_{A^{\pm1}(K)}), \label{eq:identification1-2}
\end{align*}
respectively, where $L_{K}$ and $L_{A^{\pm1}(K)}$ are the corresponding surgery duals. 
\end{rem}
\begin{figure}[h]
\centering
\includegraphics[scale=0.715]{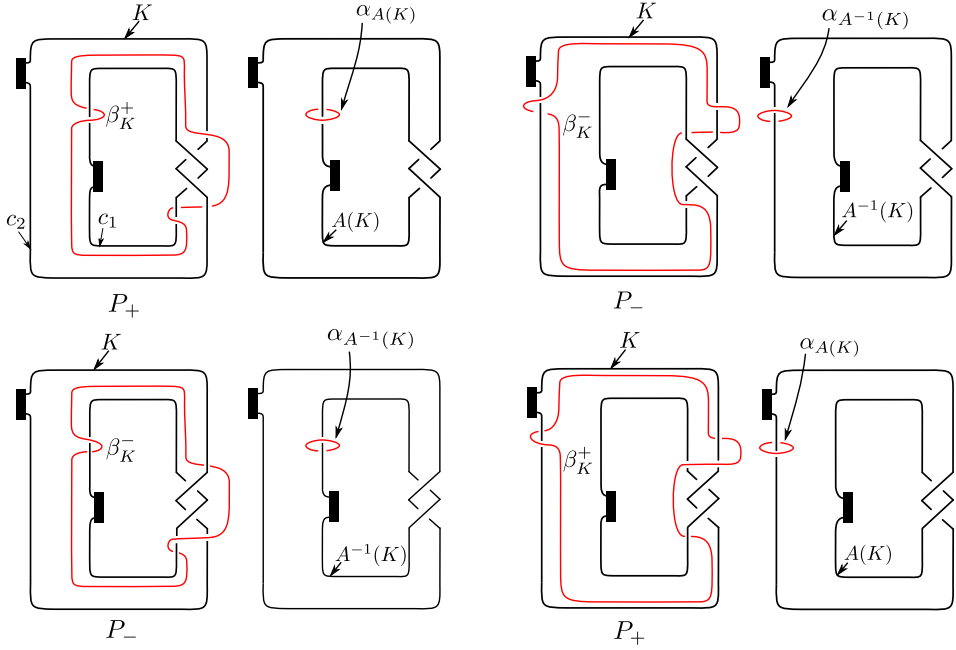}
\caption{(color online) From a special annulus presentation $(A,b)$ of a knot $K$ to dualizable patterns $P_{+}$ and $P_{-}$ given by $K\subset \mathbf{S}^{3}\setminus \nu(\beta_{K}^{\pm})=V_{\pm}$
}\label{figure:annulus-dualizable}
\end{figure}
%
%
%
\section{Operation $(\ast m)$ and dualizable pattern}\label{sec:dualizable2}
By Theorems~\ref{thm:AJLO1} and \ref{thm:MP}, for a knot $K$ with a special annulus presentation $(A,b)$, we have 
\[
 X_{\tau_{m}(P_{+}^{\ast})(U)}(m)\cong X_{P_{+}(U)}(m)= X_{K}(m)\cong X_{T_{m}(A(K))}(m), 
\]
where $P_{+}$ is the dualizable pattern obtained from $K$ as in Section~\ref{sec:annulus-dualizable}. 
Hence, it is a natural question whether $\tau_{m}(P_{+}^{\ast})(U)$ is isotopic to $T_{m}(A(K))$ or not. 
Proposition~\ref{prop:Miller-Piccirillo} implies that the answer is ``yes" if $m=0$. 
The following theorem gives the affirmative answer to this question for any $m\in\mathbf{Z}$. 
\begin{thm}\label{thm:star-dualizable}
Let $K$ be a knot with a special annulus presentation $(A,b)$. 
Let $P_{+}$ be the dualizable pattern obtained from $K$ as in Section~\ref{sec:annulus-dualizable}. 
Then, we obtain $\tau_{m}(P_{+}^{\ast})(U)=T_{m}(A(K))$ for any $m\in\mathbf{Z}$. 
\end{thm}
%

%
%
Miller and Piccirillo \cite[Proposition~5.3]{Miller-Piccirillo} proved Theorem~\ref{thm:star-dualizable} for $m=0$. 
We can prove Theorem~\ref{thm:star-dualizable} by extending Miller and Piccirillo's proof as follows. 
\begin{proof}
Let $L_{T_{m}(A(K))}^{(m)}\subset M_{T_{m}(A(K))}(m)$ be the surgery dual to $T_{m}(A(K))$. 
Let $\alpha_{T_{m}(A(K))}\subset \mathbf{S}^{3}\setminus \nu(T_{m}(A(K)))$ be a meridian of $T_{m}(A(K))$. 
Then, we can regard $\alpha_{T_{m}(A(K))}$ as a curve in $M_{T_{m}(A(K))}(m)$ by using the following identification
\begin{align}
\mathbf{S}^{3}\setminus \nu(T_{m}(A(K)))=M_{T_{m}(A(K))}(m)\setminus \nu (L_{T_{m}(A(K))}^{(m)}). \label{eq:first}
\end{align}
Since $\alpha_{T_{m}(A(K))}$ is isotopic to $L_{T_{m}(A(K))}^{(m)}$ in $M_{T_{m}(A(K))}(m)$, we have 
\begin{align}
M_{T_{m}(A(K))}(m)\setminus \nu (L_{T_{m}(A(K))}^{(m)})\cong M_{T_{m}(A(K))}(m)\setminus \nu (\alpha_{T_{m}(A(K))}). \label{eq:second}
\end{align}
\par 
Let $\beta_{K}^{+}\subset \mathbf{S}^3\setminus \nu(K)$ be the curve given in Section~\ref{sec:annulus-dualizable} (see also Figure~\ref{figure:annulus-dualizable}). 
We can also regard $\beta_{K}^{+}$ as a curve in $M_{K}(m)$ under the identification $\mathbf{S}^3\setminus \nu(K)\cong M_{K}(m)\setminus \nu(L_{K}^{(m)})$, where $L_{K}^{(m)}$ is the surgery dual to $K$. 
Then, we can check that $\psi_{m}(\beta_{K}^{+})=\alpha_{T_{m}(A(K))}$, where $\psi_{m}\colon M_{K}(m)\rightarrow M_{T_{m}(A(K))}(m)$ is given in Figure~\ref{figure:AJLO-homeo}. 
Hence, we obtain 
\begin{align}
M_{T_{m}(A(K))}(m)\setminus \nu (\alpha_{T_{m}(A(K))}) \label{eq:third}
&\cong M_{K}(m)\setminus \nu (\beta_{K}^{+})\\ \nonumber
&\cong M_{K\cup \alpha_{K}}(0,-1/m)\setminus \nu (\beta_{K}^{+})\\ \nonumber
&\cong \mathbf{S}^3\setminus \nu(K\cup \alpha_{K}\cup \beta_{K}^{+})\cup\bigsqcup_{i=0,1}(S_{i}^{1}\times D_{i}^{2}), \nonumber
\end{align}
where the last (small) union is given by identifying $\partial D_{0}^{2}$ with $0$-framing of $K$ and $\partial D_{1}^{2}$ with $-1/m$-framing of $\alpha_{K}$. 
\par 
Recall that the solid torus $V_{+}$ containing $P_{+}$ is given by $V_{+}=\mathbf{S}^{3}\setminus \nu(\beta_{K}^{+})$. 
Since, the $0$-framing of $K$ is viewed as $\lambda_{P_{+}}$ and $\alpha_{K}$ is viewed as $\mu_{P_{+}}$ in $V_{+}$, we have 
\begin{align}
\label{eq:fourth}
&\mathbf{S}^3\setminus \nu(K\cup \alpha_{K}\cup \beta_{K}^{+})\cup\bigsqcup_{i=0,1}(S_{i}^{1}\times D_{i}^{2}) \\
&\cong ((V_{+}\setminus \nu(P_{+}))\setminus \nu(\mu_{P_{+}}))\cup\bigsqcup_{i=0,1}(S_{i}^{1}\times D_{i}^{2}), \nonumber
\end{align}
where the last (small) union is given by identifying $\partial D_{0}^{2}$ with $\lambda_{P_{+}}$ and $\partial D_{1}^{2}$ with $-1/m$-framing of $\mu_{P_{+}}$. 
By the dualizability of $P_{+}$, we obtain 
\begin{align}
\label{eq:last}
&((V_{+}\setminus \nu(P_{+}))\setminus \nu(\mu_{P_{+}}))\cup\bigsqcup_{i=0,1}(S_{i}^{1}\times D_{i}^{2}) \\ \nonumber 
&\cong ((V_{+}^{\ast}\setminus \nu (P^{\ast}_{+}))\setminus \nu (\mu_{V^{\ast}_{+}}))\cup\bigsqcup_{i=0,1}(S_{i}^{1}\times D_{i}^{2})\\ \nonumber
&\cong (V_{+}^{\ast}\setminus \nu (\tau_{m}(P^{\ast}_{+})))\cup(S_{0}^{1}\times D_{0}^{2})\\ \nonumber
&\cong \mathbf{S}^3\setminus \nu(\tau_{m}(P^{\ast}_{+})(U)), \nonumber
\end{align}
where the last union is given by identifying $\partial D_{0}^{2}$ with $\lambda_{V_{+}^{\ast}}$. 
By (\ref{eq:first})--(\ref{eq:last}) and the Knot Complement Theorem, we obtain $\tau_{m}(P_{+}^{\ast})(U)=T_{m}(A(K))$. 
\end{proof}
\begin{rem}\label{rem:mirror}
Let $K$ be a knot with a special annulus presentation $(A,b)$. 
Let $\overline{K}$ be the mirror image of $K$ and $(\overline{A},\overline{b})$ be the special annulus presentation of $\overline{K}$ obtained from $(A,b)$ by taking mirror image. 
Let $\gamma_{A^{-1}(K)}\subset \mathbf{S}^{3}\setminus \nu(A^{-1}(K))$ be the mirror image of $\gamma_{\overline{A}(\overline{K})}\subset \mathbf{S}^{3}\setminus \nu(\overline{A}(\overline{K}))$ (see also Figure~\ref{figure:dual}). 
Denote the knot obtained from $A^{-1}(K)$ by twisting $m$ times along $\gamma_{A^{-1}(K)}$ by $T_{m}(A^{-1}(K))$. 
Then, by the similar discussion to Theorem~\ref{thm:star-dualizable}, we see that $\tau_{m}(P^{\ast}_{-})(U)=T_{m}(A^{-1}(K))$ for any $m\in \mathbf{Z}$. 
\end{rem}
We see that $A(K)\subset \mathbf{S}^{3}\setminus \nu(\gamma_{A(K)})=V'_{+}$ also gives a dualizable pattern, where the parameter of $V'_{+}\cong \mathbf{S}^1\times D^2$ is given by the standard way. 
Denote it by $P_{+}'$. 
It is easy to see that $\tau_{m}(P_{+}')(U)=T_{m}(A(K))=\tau_{m}(P_{+}^{\ast})(U)$ for any $m\in\mathbf{Z}$. 
So we can consider the question which asks whether $P'_{+}$ is equal to $P_{+}^{\ast}$ as a pattern. 
We can give the affirmative answer to the question as follows. 
\begin{thm}\label{thm:main2}
Let $K$ be a knot with a special annulus presentation $(A,b)$. 
Let $P_{+}'\subset V'_{+}$ be the dualizable pattern as above, and let $P_{+}^{\ast}\subset V^{\ast}_{+}$ be the dualizable pattern obtained from $K$ as in Section~\ref{sec:annulus-dualizable}. 
Then, for any $m\in \mathbf{Z}$, there is an orientation-preserving homeomorphism $h\colon V'_{+}\rightarrow V^{\ast}_{+}$ such that 
\begin{itemize}
\item $h(\tau_{m}(P'_{+}))=\tau_{m}(P_{+}^{\ast})$, 
\item $h(\lambda_{V'_{+}})=\lambda_{V^{\ast}_{+}}$ and $h(\mu_{V'_{+}})=\mu_{V^{\ast}_{+}}$. 
\end{itemize}
Namely, $P'_{+}=P^{\ast}_{+}$ as patterns. 
\end{thm}
\begin{proof}
By the definition of the operation $(\ast m)$, we see that  
\begin{align}
&M_{T_{m}(A(K))}(m)\setminus \nu (\alpha_{T_{m}(A(K))}) \label{eq:add1}\\\nonumber
&\cong M_{A(K)\cup \gamma_{A(K)}}(0,-1/m)\setminus \nu (\alpha_{A(K)}) \\\nonumber
&\cong \mathbf{S}^3\setminus \nu(A(K)\cup \gamma_{A(K)}\cup \alpha_{A(K)})\cup\bigsqcup_{i=0,1}(S_{i}^{1}\times D_{i}^{2}), \nonumber
\end{align}
where the last (small) union is given by identifying $\partial D_{0}^{2}$ with $0$-framing of $A(K)$ and $\partial D_{1}^{2}$ with $-1/m$-framing of $\gamma_{(A(K))}$. 
Since $\alpha_{A(K)}$ is isotopic to the surgery dual to $A(K)$, we have  
\begin{align}
&\mathbf{S}^3\setminus \nu(A(K)\cup \gamma_{A(K)}\cup \alpha_{A(K)})\cup\bigsqcup_{i=0,1}(S_{i}^{1}\times D_{i}^{2}) \label{eq:add2}\\\nonumber
&\cong \mathbf{S}^3\setminus \nu(A(K)\cup \gamma_{A(K)})\cup(S_{1}^{1}\times D_{1}^{2})\\\nonumber
&=(V'_{+}\setminus \nu (\tau_{m}(P_{+}')))\cup (S_{2}^{1}\times D_{2}^{2}),\nonumber
\end{align}
where the last union is given by identifying $\partial D_{2}^{2}$ with $\lambda_{V'_{+}}$. 
By considering the composition of (\ref{eq:add2}), (\ref{eq:add1}), (\ref{eq:third}), (\ref{eq:fourth}) and (\ref{eq:last}) we obtain an orientation-preserving homeomorphism
\[
\overline{h}\colon (V_{+}'\setminus \nu(\tau_{m}(P_{+}')))\cup (S_{2}^{1}\times D_{2}^{2})\rightarrow 
(V^{\ast}_{+}\setminus \nu(\tau_{m}(P^{\ast}_{+})))\cup(S_{0}^{1}\times D_{0}^{2}). 
\]
Then, we can check that 
\begin{itemize}
\item $\overline{h}(\lambda_{\tau_{m}(P'_{+})})=\lambda_{\tau_{m}(P_{+}^{\ast})}$, 
\item $\overline{h}(\lambda_{V'_{+}})=\lambda_{V^{\ast}_{+}}$ and $\overline{h}(\mu_{V'_{+}})=\mu_{V^{\ast}_{+}}$, 
\item $\overline{h}(S_{2}^{1}\times D_{2}^{2})=S_{0}^{1}\times D_{0}^{2}$.    
\end{itemize}
Hence, $\overline{h}$ induces a desired homeomorphism. 
\end{proof}
\begin{rem}\label{rem:mirror2}
Similarly, we can define $P'_{-}$ as $A^{-1}(K)\subset \mathbf{S}^{3}\setminus \nu(\gamma_{A^{-1}(K)})=V'_{-}$ (see also Remark~\ref{rem:mirror}). 
By the same discussion as the proof of Theorem~\ref{thm:main2}, we see that there is an orientation-preserving homeomorphism $h\colon V'_{-}\rightarrow V^{\ast}_{-}$ which satisfies 
$h(\tau_{m}(P'_{-}))=\tau_{m}(P_{-}^{\ast})$, 
$h(\lambda_{V'_{-}})=\lambda_{V^{\ast}_{-}}$ and $h(\mu_{V'_{-}})=\mu_{V^{\ast}_{-}}$. 
\end{rem}
\begin{rem}
We see that Theorem~\ref{thm:main2} induces Theorem~\ref{thm:star-dualizable} since $\tau_{m}(P_{+}^{\ast})(U)=\tau_{m}(P_{+}')(U)$ by Theorem~\ref{thm:main2} and $\tau_{m}(P_{+}')(U)=T_{m}(A(K))$ by the definition of $P'_{+}$.  
\end{rem}
\par
By Theorem~\ref{thm:main2} and Remark~\ref{rem:mirror2}, we can draw the duals $P^{\ast}_{\pm}$ to $P_{\pm}$ as in Figure~\ref{figure:dual}, where $P_{\pm}$ are the dualizable patterns obtained from a knot $K$ with a special annulus presentation $(A,b)$ as in Section~\ref{sec:annulus-dualizable}. 
\begin{figure}[h]
\centering
\includegraphics[scale=0.7]{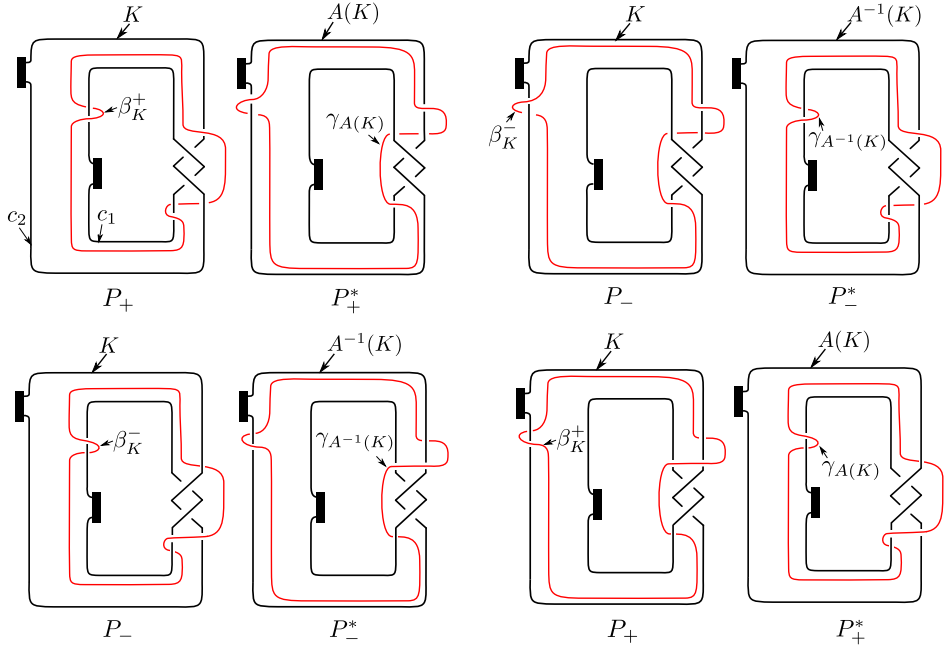}
\caption{(color online) The dualizable patterns $P_{\pm}\subset V_{\pm}=\mathbf{S}^{3}\setminus \nu(\beta_{K}^{\pm})$ and $P_{\pm}^{\ast}\subset V_{\pm}^{\ast}= \mathbf{S}^{3}\setminus \nu(\gamma_{A^{\pm1}(K)})$ }
\label{figure:dual}
\end{figure}
%
%
\section{Flipped annulus twist and operation $(\ast\pm4)$}\label{sec:Teragaito}
In \cite{Teragaito}, Teragaito gave the first example of a Seifert fibered manifold which is represented by the same integral surgery on infinitely many hyperbolic knots. 
In the work, Teragaito used a presentation of $9_{42}$, which is almost the same as a special annulus presentation but does not satisfy the last condition: $A\cup b$ is an immersion of an orientable surface. 
Teragaito explained that, for a knot with such a presentation, we obtain a family of knots admitting the same $4$-surgery (not $0$-surgery) by annulus twists along (a shrunken annulus of) the annulus. 
It has been known that such knots have the same $4$-trace (see \cite[Theorem~2.8]{AJOT}). 
\par 
In this section, we prove that the above phenomenon can be explained in terms of the operation $(\ast 4)$. 
\subsection{Flipped annulus twist}
Let $A\subset \mathbf{S}^3$ be an embedded annulus with ordered boundary $\partial A=c_1\cup c_2$. 
We suppose that $A$ is unknottend and $\operatorname{lk}(c_1,c_2)=\pm1$, where we give $c_1$ and $c_2$ parallel orientations. 
Then, an {\it $n$-fold flipped annulus twist along $A$} is to apply $(-lk(c_1,c_2)+1/n)$-surgery along $c_{1}$ and $(-lk(c_1,c_2)-1/n)$-surgery along $c_{2}$ (compare with Section~\ref{sec:annulus-twist-pre}). 
\par
Let $K$ be a knot with a special annulus presentation $(A,b)$. 
Then, by $A_{f}^{n}(K)$, 
we denote the knot obtained from $K$ by the $n$-fold flipped annulus twist along $A'$, where $A'$ is a shrunken annulus given in Section~\ref{sec:annulus-twist-pre}. 
For simplicity, we also denote $A_{f}^{1}(K)$ by $A_{f}(K)$. 
We also see $A_{f}^{n}(K)$ as follows: 
After ``flipping" $c_1$ (or $c_{2}$) as in Figure~\ref{figure:annulus-flip}, we find another annulus $A_{f}$. 
Then, by using \cite[Lemma~7.15]{abe-tagami3}, we see that $A_{f}^{n}(K)$ is obtained from $K$ by applying the $n$-fold annulus twist along $A'_{f}$, where $A'_{f}$ is a shrunken annulus of $A_{f}$. 
Remark that $(A_{f}, b)$ is not an annulus presentation any more since $A_{f}\cup b$ is an immersion of a non-orientable surface. 
\begin{figure}[h]
\centering
\includegraphics[scale=0.7]{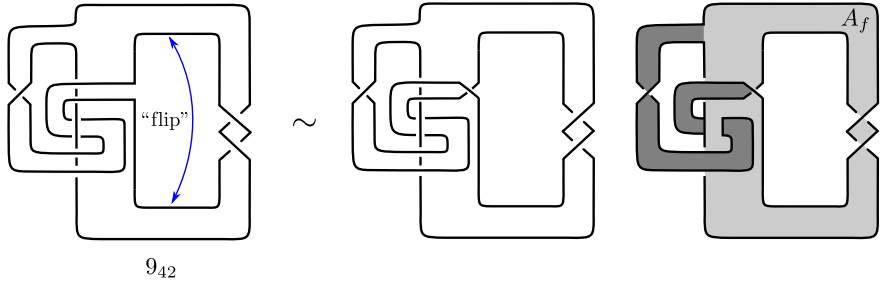}
\caption{(color online) An annulus presentation of $9_{42}$ (left). After ``flipping" $c_1$, we find a new annulus $A_{f}$. }
\label{figure:annulus-flip}
\end{figure}
%
\subsection{Relation to the operation $(\ast \pm4)$}
Teragaito \cite[Proposition~2.1]{Teragaito} proved that there is an orientation-preserving homeomorphism $M_{K}(r)\rightarrow M_{A^{n}_{f}(K)}(r)$, where 
$
r=-4\operatorname{lk}(c_1,c_2) \in \{\pm 4\}. 
$
Denote this homeomorphism by 
\[
\phi_{n}^{f}\colon M_{K}(r)\rightarrow M_{A^{n}_{f}(K)}(r). 
\]
For a sketch of the proof, see Figure~\ref{figure:Teragaito_homeo}. 
Then, we notice that 
\begin{align}
\phi_{\pm1}^{f}(\beta^{\mp}_{K})=\alpha_{A^{\pm1}_{f}(K)}, \label{eq:meridian}
\end{align}
where $\alpha_{A^{\pm1}_{f}(K)}$ is a meridian of $A^{\pm1}_{f}(K)$ and we regard $\beta^{\pm}_{K}$ and $\alpha_{A^{\pm1}_{f}(K)}$ as curves in $M_{K}(r)$ and $M_{A^{n}_{f}(K)}(r)$, respectively (by using the same discussion in Remark~\ref{rem:regarding}). 
We have seen that 
$M_{T_{r}(A(K))}(r)\cong M_{K}(r)\cong M_{A^{-1}_{f}(K)}(r)$. 
Moreover, we can prove that
\begin{align}
T_{r}(A^{\mp1}(K))= A^{\pm1}_{f}(K). \label{eq:flip}
\end{align}
In fact, by replacing $\psi_{m}$ with $\phi_{\pm1}^{f}$ and $\beta_{K}^{+}$ with $\beta_{K}^{\mp}$ in the proof of Theorem~\ref{thm:star-dualizable}, we see that 
$
\mathbf{S}^{3}\setminus \nu(A_{f}^{\pm1}(K))
\cong \mathbf{S}^{3}\setminus \nu(\tau_{r}(P^{\ast}_{\mp})(U))
\cong \mathbf{S}^{3}\setminus \nu(T_{r}(A^{\mp1}(K))). 
$
By the Knot Complement Theorem, we obtain Equation~(\ref{eq:flip}). 
%
%
As a consequence, we obtain the following. 
\begin{thm}\label{thm:teragaito-star-m}
Let $K$ be a knot with a special annulus presentation $(A,b)$ with $\partial A=c_1\cup c_2$. 
Then we obtain 
\[
T_{r}(A^{\mp1}(K))= A^{\pm1}_{f}(K), 
\]
where $r=-4\operatorname{lk}(c_1, c_2)$, and we give $c_1$ and $c_2$ parallel orientations. 
\end{thm}
\begin{figure}[h]
\centering
\includegraphics[scale=0.72]{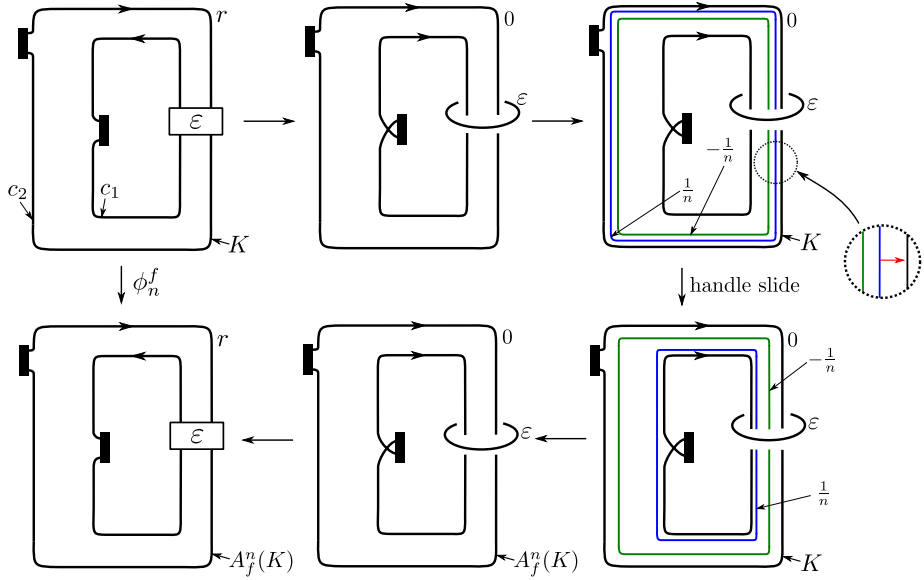}
\caption{(color online) . The homeomorphism $\phi_{n}^{f}\colon M_{K}(r)\rightarrow M_{A^{n}_{f}(K)}(r)$, where $\varepsilon\in\{\pm1\}$ and 
$r=-4\varepsilon$. The box with $\varepsilon$ represents $\varepsilon$-full-twist. For convenience, we draw an orientation of the knot (not $c_1$ and $c_2$). }
\label{figure:Teragaito_homeo}
\end{figure}
\begin{rem}
In private communication, Tetsuya Abe commented that $T_{4}(A(9_{42}))$ and $A^{-1}_{f}(9_{42})$ may be equivalent because of computational calculations. 
Theorem~\ref{thm:teragaito-star-m} is inspired by the comment. 
\end{rem}
%
\section{Discussions}
\subsection{Naturality}
Let $K$ be a knot with a special annulus presentation $(A,b)$. 
Then, we obtain a dualizable pattern $P_{+}$ as in Section~\ref{sec:annulus-dualizable}. 
Put $\check{K}=A(K)$ and give the natural annulus presentation $(\check{A},\check{b})$ of $\check{K}$ from $(A,b)$. 
Then we obtain another dualizable pattern $\check{P}_{-}$ from $\check{K}$ as in Section~\ref{sec:annulus-dualizable}. 
We see that these patterns satisfy $P_{+}(U)=K$, $P_{+}^{\ast}(U)=A(K)$, $\check{P}_{-}(U)=A(K)$ and $\check{P}_{-}^{\ast}(U)=K$. 
More strongly, Theorem~\ref{thm:main2} and Figure~\ref{figure:dual} imply that $P_{+}=\check{P}_{-}^{\ast}$ and $P_{+}^{\ast}=\check{P}_{-}$. 
\par
Let $SAP$ be the set of special annulus presentations and $DP$ be the set of unoriented patterns which are dualizable after giving some orientation. 
Then, by the above discussion, we obtain the following commutative diagram: 
%
\[
\begin{CD}
SAP @>\pm>> DP\\
@V{A^{\pm1}}VV @VV{ \ast}V\\
SAP @>\mp>> DP
\end{CD}
\]
where 
\begin{itemize}
\item $A^{\pm1}\colon SAP\rightarrow SAP$ is the map induced by $\pm1$-fold annulus twist, 
\item $\pm\colon SAP\rightarrow DP$ is given by $(A,b)\mapsto P_{\pm}$ as in Section~\ref{sec:annulus-dualizable} and 
\item $\ast\colon DP\rightarrow DP$ is given by $P\mapsto P^{\ast}$.  
\end{itemize}

\subsection{Generalization}
For more general setting, we obtain the following result by the proof of Theorem~\ref{thm:star-dualizable}. 
Remark that, in Theorem~\ref{thm:general} below, we regard curves $\beta$ and $\alpha_{K_{2}}$ as curves in $M_{K_{1}}(m)$ and $M_{K_{2}}(m)$, respectively, under the identification $\mathbf{S}^3\setminus \nu(K_{i})=M_{K_{i}}(m)\setminus \nu(L_{K_{i}})$, where $L_{K_{i}}$ is the surgery dual. 
\begin{thm}\label{thm:general}
Let $K_{1}$ and $K_{2}$ be knots in $\mathbf{S}^3$. 
Let $\beta\subset \mathbf{S}^3\setminus \nu(K_{1})$ be an unknot. 
Let $\alpha_{K_{2}}\subset \mathbf{S}^3\setminus \nu(K_{2})$ be a meridian of $K_{2}$. 
Suppose that $P=K_{1}\subset \mathbf{S}^{3}\setminus \nu(\beta)$ gives a dualizable pattern. 
Then, if there is an orientation-preserving homeomorphism $\phi\colon M_{K_{1}}(m)\rightarrow M_{K_{2}}(m)$ such that $\phi(\beta)=\alpha_{K_{2}}$, we have $\tau_{m}(P^{\ast})(U)=K_{2}$. 
Moreover $\phi$ extends to a diffeomorphism $\Phi\colon X_{K_1}(m)\rightarrow X_{K_{2}}(m)$. 
\end{thm}
The last claim follows from the same discussion as Theorem~\ref{thm:MP}. 
\begin{question}
Let $\phi\colon M_{K_{1}}(m)\rightarrow M_{K_{2}}(m)$ be an orientation-preserving homeomorphism. 
Then, when is there an unknot $\beta\subset \mathbf{S}^3\setminus \nu(K_{1})$ which satisfies the condition of Theorem~\ref{thm:general}? 
Moreover, if exists, is such $\beta$ unique up to isotopy in $\mathbf{S}^3\setminus \nu(K_{1})$?
\end{question}
%


%
\ \\
\noindent
\textbf{Acknowledgements.}
The author was supported by JSPS KAKENHI Grant number JP18K13416. 

\bibliographystyle{amsplain}
\bibliography{tagami}
\end{document}